\documentclass{amsart}
\usepackage{mathrsfs}
\usepackage{mathrsfs}
\usepackage{amsfonts}
\usepackage{cases}
\usepackage{latexsym}
\usepackage{amsmath}
\usepackage[all]{xy}
\usepackage{stmaryrd}
\usepackage{amsfonts}
\usepackage{amsmath,amssymb,amscd,bbm,amsthm,mathrsfs,dsfont}

\usepackage{fancyhdr}
\usepackage{amsxtra,ifthen}
\usepackage{verbatim}

\numberwithin{equation}{section}

\theoremstyle{plain}
\newtheorem{theorem}{Theorem}[section]

\newtheorem{proposition}[theorem]{Proposition}
\newtheorem{corollary}[theorem]{Corollary}

\theoremstyle{definition}
\newtheorem{definition}[theorem]{Definition}
\newtheorem{example}[theorem]{Example}

\newtheorem{question}[theorem]{Question}

\begin{document}

\title[Jordan Derivations and Antiderivations]
{Jordan Derivations and Antiderivations of Generalized Matrix
Algebras}

\author{Yanbo Li, Leon van Wyk and Feng Wei}

\address{Li: Department of Information and Computing Sciences, Northeastern
University at Qinhuangdao, Qinhuangdao, 066004, P.R. China}

\email{liyanbo707@gmail.com}

\address{van Wyk: Department of Mathematics, Stellenbosch University,
 Private Bag X1, Matieland 7602, Stellenbosch, South Africa}

\email{LvW@sun.ac.za}

\address{Wei: School of Mathematics, Beijing Institute of
Technology, Beijing, 100081, P. R. China}

\email{daoshuo@hotmail.com}

\begin{abstract}
Let $\mathcal{G}=
\left[\smallmatrix A & M\\
N & B \endsmallmatrix \right]$ be a generalized matrix algebra
defined by the Morita context $(A, B, _AM_B, _BN_A, \Phi_{MN},
\Psi_{NM})$. In this article we mainly study the question of whether
there exist proper Jordan derivations for the generalized matrix
algebra $\mathcal{G}$. It is shown that if one of the bilinear
pairings $\Phi_{MN}$ and $\Psi_{NM}$ is nondegenerate, then every
antiderivation of $\mathcal{G}$ is zero. Furthermore, if the
bilinear pairings $\Phi_{MN}$ and $\Psi_{NM}$ are both zero, then
every Jordan derivation of $\mathcal{G}$ is the sum of a derivation
and an antiderivation. Several constructive examples and
counterexamples are presented.
\end{abstract}



\thanks{This work is partially supported by the
National Natural Science Foundation of China (Grant No. 10871023).}
\date{September 13, 2011}

\maketitle
\section{Introduction}
\label{xxsec1}

Let us begin with the definition of generalized matrix algebras
given by a Morita context. Let $\mathcal{R}$ be a commutative ring
with identity. A Morita context consists of two
$\mathcal{R}$-algebras $A$ and $B$, two bimodules $_AM_B$ and
$_BN_A$, and two bimodule homomorphisms called the pairings
$\Phi_{MN}: M\underset {B}{\otimes} N\longrightarrow A$ and
$\Psi_{NM}: N\underset {A}{\otimes} M\longrightarrow B$ satisfying
the following commutative diagrams:
$$
\xymatrix{ M \underset {B}{\otimes} N \underset{A}{\otimes} M
\ar[rr]^{\hspace{8pt}\Phi_{MN} \otimes I_M} \ar[dd]^{I_M \otimes
\Psi_{NM}} && A
\underset{A}{\otimes} M \ar[dd]^{\cong} \\  &&\\
M \underset{B}{\otimes} B \ar[rr]^{\hspace{10pt}\cong} && M }
\hspace{4pt}{\rm and}\hspace{4pt} \xymatrix{ N \underset
{A}{\otimes} M \underset{B}{\otimes} N
\ar[rr]^{\hspace{8pt}\Psi_{NM}\otimes I_N} \ar[dd]^{I_N\otimes
\Phi_{MN}} && B
\underset{B}{\otimes} N \ar[dd]^{\cong}\\  &&\\
N \underset{A}{\otimes} A \ar[rr]^{\hspace{10pt}\cong} && N
\hspace{2pt}.}
$$
Let us write this Morita context as $(A, B, _AM_B, _BN_A, \Phi_{MN},
\Psi_{NM})$. If $(A, B, _AM_B,$ $ _BN_A,$ $ \Phi_{MN}, \Psi_{NM})$
is a Morita context, then the set
$$
\left[
\begin{array}
[c]{cc}%
A & M\\
N & B\\
\end{array}
\right]=\left\{ \left[
\begin{array}
[c]{cc}%
a& m\\
n & b\\
\end{array}
\right] \vline a\in A, m\in M, n\in N, b\in B \right\}
$$
form an $\mathcal{R}$-algebra under matrix-like addition and
matrix-like multiplication, where We assume that at least one of the
two bimodules $M$ and $N$ is distinct from zero. Such an
$\mathcal{R}$-algebra is called a \textit{generalized matrix
algebra} of order 2 and is usually denoted by $\mathcal{G}=
\left[\smallmatrix A & M\\
N & B \endsmallmatrix \right]$. This kind of algebra was first
introduced by Morita in \cite{Morita}, where the author investigated
Morita duality theory of modules and its applications to Artinian
algebras.

Let $\mathcal{R}$ be a commutative ring with identity, $A$ be a
unital algebra over $\mathcal{R}$ and $\mathcal{Z}(A)$ be the center
of $A$. Recall that an $\mathcal{R}$-linear map $\Theta_{\rm d}$
from $A$ into itself is called a \textit{derivation} if $\Theta_{\rm
d}(ab)=\Theta_{\rm d}(a)b+a\Theta_{\rm d}(b)$ for all $a, b\in A$.
Further, an $\mathcal{R}$-linear map $\Theta_{\rm Jord}$ from $A$
into itself is called a \textit{Jordan derivation} if $\Theta_{\rm
Jord}(a^2)=\Theta_{\rm Jord}(a)a+a\Theta_{\rm Jord}(a)$ for all
$a\in A$. Every derivation is obviously a Jordan derivation. The
inverse statement is not true in general. Those Jordan derivations
which are not derivations are said to be \textit{proper}. An
$\mathcal{R}$-linear map $\Theta_{\rm antid}$ from $A$ into itself
is called an \textit{antiderivation} if $\Theta_{\rm
antid}(ab)=\Theta_{\rm antid}(b)a+b\Theta_{\rm antid}(a)$ for all
$a, b\in A$.

In 1957 Herstein \cite{Herstein} proved that every Jordan derivation
from a prime ring of characteristic not $2$ into itself is a
derivation. This result has been generalized to different rings and
algebras in various directions (see e.g. \cite{Benkovic, Bresar1,
Bresar2, Bresar4, HanWei, HouQi, LiBenkovic, XiaoWei1, ZhangYu} and
references therein). Zhang and Yu \cite{ZhangYu} showed that every
Jordan derivation on a triangular algebra is a derivation. Xiao and
Wei \cite{XiaoWei1} extended this result to the higher case and
obtained that any Jordan higher derivation on a triangular algebra
is a higher derivation. Johnson \cite{Johnson} considered a more
challenging question for which Banach algebras $A$ there are no
proper Jordan derivations from $A$ into an arbitrary Banach
$A$-bimodule $M$. It turned out that this is true for some important
classes of algebras (in particular, for the algebra of all $n\times
n$ complex matrices). Motivated by Johnson's work, Benkovic
investigated the structure of Jordan derivations from the upper
triangular matrix algebra $\mathcal{T}_n(\mathcal{R})$ into its
bimodule and proved that every Jordan derivation from
$\mathcal{T}_n(\mathcal{R})$ into its bimodule is the sum of a
derivation and an antiderivation. Recently, Li, Xiao and Wei
\cite{LiXiao, LiWei, XiaoWei2} jointly studied linear maps of
generalized matrix algebras, such as derivations, Lie derivations,
commuting maps and semicentralizing maps. Our main purpose is to
develop the theory of linear maps of triangular algebras to the case
of generalized matrix algebras, which has a much broader background.
People pay much less attention to linear maps of generalized matrix
algebras, to the best of our knowledge there are fewer articles
dealing with linear maps of generalized matrix algebras except for
\cite{BobocDascalescuvanWyk, LiXiao, LiWei, XiaoWei2}.

The problem that we address in this article is to study whether
there exist proper Jordan derivations for generalized matrix
algebras. The outline of this article is as follows. The second
section presents two basic examples of generalized matrix algebras
which we will revisit later. In the third section we describe the
general form of Jordan derivations and antiderivations on
generalized matrix algebras. We observe that any antiderivation on a
class of generalized matrix algebra is zero (see Proposition
\ref{xxsec3.10}). Furthermore, it is shown that every Jordan
derivation on another class of generalized matrix algebras is the
sum of a derivation and an antiderivation (see Theorem
\ref{xxsec3.11}).

\section{Examples of Generalized Matrix Algebras}\label{xxsec2}

We have presented many examples of generalized matrix algebras in
\cite{LiWei}, such as standard generalized matrix algebras and
quasi-hereditary algebras, generalized matrix algebras of order $n$,
inflated algebras, upper and lower triangular matrix algebras, block
upper and lower triangular matrix algebras, nest algebras. For later
discussion convenience, we have to give another two new generalized
matrix algebras.

\subsection{Generalized matrix algebras from smash product algebras}
\label{xxsec2.1}

Let $H$ be a finite dimensional Hopf algebra over a field
$\mathbb{K}$ with comultiplication $\Delta: H\longrightarrow
H\bigotimes H$, counit $\varepsilon: H\longrightarrow \mathbb{K}$
and antipode $S: H\longrightarrow H$. Clearly, $S$ is bijective.
Moreover, the space of left integrals $\int_H l=\left\{x\in
H|hx=\varepsilon(h)x, \forall h\in H\right\}$ is one-dimensional. We
substitute the ``sigma notation" for $\Delta$ in the present
article. Now assume that $A$ is an $H$-module algebra, that is, $A$
is a $\mathbb{K}$-algebra which is a left $H$-module, such that
\begin{enumerate}
\item[(1)] $h\cdot(ab)=\underset{(h)}{\sum}(h_1\cdot a) (h_2\cdot b)$
and

\item[(2)] $h\cdot 1_A=\varepsilon(h)1_A$.
\end{enumerate}
for all $h\in H, a, b\in A$. Then the \textit{smash product algebra}
$A\# H$ is defined as follows, for any $a, b\in A, h, k\in H$:
\begin{enumerate}
\item[(1)] as a $\mathbb{K}$-space, $A\#H=A\otimes H$.
We write $a\#h$ for the element $a\otimes h$

\item[(2)] multiplication is given by $(a\#h)(b\#k)=\underset{(h)}{\sum} a(h_1\cdot b)\#h_2k$.
\end{enumerate}
The \textit{invariants subalgebra} of $H$ on $A$ is the set
$A^H=\{x\in A|h\cdot x=\varepsilon(h)x, \forall h\in H\}$. $A$ is a
left $A\#H$-module in the standard way, that is
$$
a\#h\rightarrow b=a(h\cdot b)
$$
for all $a, b\in A$ and $h\in H$. For a given $t\in \int l$, then
$th\in \int l$ for all $h\in H$. Since $\int l$ is one-dimensional,
there exists $\alpha\in H^*$ such that $th=\alpha(h)t$ for all $h\in
H$. It is easy to see that $\alpha$ is multiplicative, and it is a
group-like element of $H^*$. Hence
$$
h^\alpha=\alpha\rightarrow h=\underset{(h)}{\sum}\alpha(h_2)h_1,
\quad \forall h\in H
$$
defines an automorphism on $H$. Thus $A$ is a right $A\#H$-module
via
$$
a\leftarrow b\#h=S^{-1}h^\alpha\cdot (ab), \quad \forall a\in A,
b\#h\in A\#H.
$$

The close relationship between $A\# H$ and $A^H$ enables us to
formalize the following generalized matrix algebra. Now $A$ is a
left (or right) $A^H$-module simply by left (or right)
multiplication. Simultaneously, $A$ is also a left (or right)
$A\#H$-module. Thus $M=_{A^H}A_{A\#H}$ and $N=_{A\#H}A_{A^H}$,
together with the maps
$$
\begin{aligned}
\Psi_{NM}: A\otimes_{A^H} A &
\longrightarrow A\#H \hspace{3pt} {\rm defined \hspace{3pt} by}\hspace{3pt} & \Psi_{NM}(a,b)&=(a\#t)(b\#1)\\
\Phi_{MN}: A\otimes_{A\#H} A^H & \longrightarrow A^H \hspace{3pt}
{\rm defined \hspace{3pt} by}\hspace{3pt} & \Phi_{NM}(a,b)&=t\cdot
(ab)
\end{aligned}
$$
give rise to a new generalized matrix algebra
$$
\mathcal{G}_{\rm SPA}=\left[
\begin{array}
[c]{cc}%
A^H & M\\
N & A\#H\\
\end{array}
\right].
$$
We refer the reader to \cite{Montgomery} about the basic properties
of $\mathcal{G}_{\rm SPA}$.

\subsection{Generalized matrix algebras from group algebras}
\label{xxsec2.2}

Let $A$ be an associative algebra over a field $\mathbb{K}$ and $G$
be a finite group of automorphisms acting on $A$. The \textit{fixed
ring} $A^G$ of the action $G$ on $A$ is the set $\left\{a\in
A|a^g=a, \forall g\in G\right\}$. The \textit{skew group algebra}
$A*G$ is the set of all formal sums $\sum _{g\in G}a_g g,
\hspace{2pt} a_g\in A$. The addition operation is componentwise and
the multiplication operation is defined distributively by the
formula
$$
ag\cdot bh=ab^{g^{-1}} gh
$$
for all $a, b\in A$ and $g, h\in G$. Clearly, $A$ is a left and
right $A^G$-module. $A$ can also be viewed as a left or right
$A*G$-module as follows: for any $x=\sum_{g\in G}a_gg\in A*G$ and
$a\in A$, we define $ x\cdot a=\sum_{g\in G} a_ga^{g^{-1}} $ and $
a\cdot x=\sum_{g\in G} (aa_g)^g. $ Then we obtain a generalized
matrix algebra
$$
\mathcal{G}_{\rm GA}=\left[
\begin{array}
[c]{cc}%
A^G & M\\
N & A*G\\
\end{array}
\right],
$$
where $M=_{A^G}A_{A*G}$ and $N=_{A*G}A_{A^G}$. The bilinear pairings
$\Phi_{MN}$ and $\Psi_{NM}$ can be established via
$$
\begin{aligned}
\Phi_{MN}: A\otimes_{A*G} A & \longrightarrow A^G \\
(x, y) &\longmapsto \sum_{g\in G}(xy)^g
\end{aligned}
$$
and
$$
\begin{aligned}
\Psi_{NM}: A\otimes_{A^G} A & \longrightarrow A*G\\
(x, y)& \longmapsto \sum_{g\in G}xy^{g^{-1}}g.
\end{aligned}
$$

\section{Jordan Derivations of Generalized Matrix Algebras}
\label{xxsec3}

Let $\mathcal{G}$ be a generalized matrix algebra of order $2$ based
on the Morita context $(A, B, _AM_B, _BN_A, \Phi_{MN}, \Psi_{NM})$
and let us denote it by
$$
\mathcal{G}:=\left[
\begin{array}
[c]{cc}%
A & M\\
N & B
\end{array}
\right].
$$
Here, at least one of the two bimodules $M$ and $N$ is distinct from
zero. The main aim of this section is to show that any Jordan
derivation on a class of generalized matrix algebras is the sum of a
derivation and an antiderivation. Our motivation originates from the
following several results. Benkovic \cite{Benkovic} proved that
every Jordan derivation from the algebra of all upper triangular
matrices into its bimodule is the sum of a derivation and an
antiderivation. Ma and Ji \cite{MaJi} extended this result to the
case of generalized Jordan derivations and obtained that every
generalized Jordan derivation from the algebra of all upper
triangular matrices into its bimodule is the sum of a generalized
derivation and an antiderivation. Zhang and Yu in \cite{ZhangYu}
showed that every Jordan derivation on a triangular algebra is a
derivation. Therefore, it is appropriate to describe and
characterize Jordan derivations of $\mathcal{G}$. Note that the
forms of derivations and Lie derivations of $\mathcal{G}$ were given
in \cite{LiWei}.

\begin{proposition}\cite[Proposition 4.2]{LiWei}\label{xxsec3.1}
An additive map $\Theta_{\rm d}$ from $\mathcal{G}$ into itself is a
derivation if and only if it has the form
$$
\begin{aligned}
& \Theta_{\rm d}\left(\left[
\begin{array}
[c]{cc}%
a & m\\
n & b\\
\end{array}
\right]\right) \\
=& \left[
\begin{array}
[c]{cc}%
\delta_1(a)-mn_0-m_0n & am_0-m_0b+\tau_2(m)\\
n_0a-bn_0+\nu_3(n) & n_0m+nm_0+\mu_4(b)\\
\end{array}
\right] , (\bigstar1)\\
& \forall \left[
\begin{array}
[c]{cc}%
a & m\\
n & b\\
\end{array}
\right]\in \mathcal{G},
\end{aligned}
$$
where $m_0\in M, n_0\in N$ and
$$
\begin{aligned} \delta_1:& A \longrightarrow A, &
 \tau_2: & M\longrightarrow M, &  \nu_3: & N\longrightarrow N , &
\mu_4: & B\longrightarrow B
\end{aligned}
$$
are all $\mathcal{R}$-linear maps satisfying the following
conditions:
\begin{enumerate}
\item[(1)] $\delta_1$ is a derivation of $A$ with
$\delta_1(mn)=\tau_2(m)n+m\nu_3(n);$

\item[(2)] $\mu_4$ is a derivation of $B$ with
$\mu_4(nm)=n\tau_2(m)+\nu_3(n)m;$

\item[(3)] $\tau_2(am)=a\tau_{2}(m)+\delta_1(a)m$ and
$\tau_2(mb)=\tau_2(m)b+m\mu_4(b);$

\item[(4)] $\nu_3(na)=\nu_3(n)a+n\delta_1(a)$ and
$\nu_3(bn)=b\nu_3(n)+\mu_4(b)n.$
\end{enumerate}
\end{proposition}

\begin{proposition}\label{xxsec3.2}
An additive map $\Theta_{\rm Jord}$ from $\mathcal{G}$ into itself
is a Jordan derivation if and only if it is of the form
$$
\begin{aligned}
& \Theta_{\rm Jord}\left(\left[
\begin{array}
[c]{cc}%
a & m\\
n & b\\
\end{array}
\right]\right)\\ =& \left[
\begin{array}
[c]{cc}%
\delta_1(a)-mn_0-m_0n+\delta_4(b) & am_0-m_0b+\tau_2(m)+\tau_3(n)\\
n_0a-bn_0+\nu_2(m)+\nu_3(n) & \mu_1(a)+n_0m+nm_0+\mu_4(b)\\
\end{array}
\right] ,(\bigstar2)\\
& \forall \left[
\begin{array}
[c]{cc}%
a & m\\
n & b\\
\end{array}
\right]\in \mathcal{G},
\end{aligned}
$$
where $m_0\in M, n_0\in N$ and
$$
\begin{aligned} \delta_1:& A \longrightarrow A, & \delta_4:& B \longrightarrow A, &  \tau_2:
& M\longrightarrow M, & \tau_3: & N\longrightarrow M,\\
\nu_2: & M\longrightarrow N, & \nu_3: & N\longrightarrow N & \mu_1:
& A\longrightarrow B & \mu_4: & B\longrightarrow B
\end{aligned}
$$
are all $\mathcal {R}$-linear maps satisfying the following
conditions:
\begin{enumerate}
\item[(1)] $\delta_1$ is a Jordan derivation on $A$ and
$\delta_1(mn)=-\delta_4(nm)+\tau_2(m)n+m\nu_3(n);$

\item [(2)] $\mu_4$ is a Jordan derivation on $B$ and
$\mu_4(nm)=-\mu_1(mn)+n\tau_2(m)+\nu_3(n)m;$

\item[(3)] $\delta_4(b^2)=2\delta_4(b)=0$ for all $ b\in B$ and
$\mu_1(a^2)=2\mu_1(a)=0$ for all $a\in A;$

\item[(4)] $\tau_2(am)=a\tau_2(m)+\delta_1(a)m+m\mu_1(a)$ and
$\tau_2(mb)=\tau_2(m)b+m\mu_4(b)+\delta_4(b)m;$

\item[(5)] $\nu_3(bn)=b\nu_3(n)+\mu_4(b)n+n\delta_4(b)$ and
$\nu_3(na)=\nu_3(n)a+n\delta_1(a)+\mu_1(a)n;$

\item[(6)] $\tau_3(na)=a\tau_3(n)$, $\tau_3(bn)=\tau_3(n)b$, $n\tau_3(n)=0$,
$\tau_3(n)n=0;$

\item[(7)] $\nu_2(am)=\nu_2(m)a$, $\nu_2(mb)=b\nu_2(m)$, $m\nu_2(m)=0$,
$\nu_2(m)m=0.$
\end{enumerate}
\end{proposition}

\begin{proof} Suppose that the Jordan derivation $\Theta_{\rm Jd}$ is
of the form
$$
\begin{aligned}
& \Theta_{\rm Jord}\left(\left[
\begin{array}
[c]{cc}%
a & m \\
n & b \\
\end{array}
\right]\right) \\
= &\left[
\begin{array}
[c]{cc}%
\delta_1(a)+\delta_2(m)+\delta_3(n)+\delta_4(b) & \tau_1(a)+\tau_2(m)+\tau_3(n)+\tau_4(b) \\
\nu_1(a)+\nu_2(m)+\nu_3(n)+\nu_4(b) & \mu_1(a)+\mu_2(m)+\mu_3(n)+\mu_4(b) \\
\end{array}
\right] ,
\end{aligned}
$$
for all $\left[\smallmatrix a & m\\
n & b
\endsmallmatrix \right]\in \mathcal{G}$, where $\delta_1,\delta_2,\delta_3,\delta_4$ are
$\mathcal{R}$-linear maps from $A, M, N, B$ to $A$, respectively;
$\tau_1,\tau_2$, $\tau_3,\tau_4$ are $\mathcal{R}$-linear maps from
$A, M, N, B$ to $M$, respectively; $\nu_1,\nu_2,\nu_3,\nu_4$ are
$\mathcal{R}$-linear maps from $A, M, N, B$ to $N$, respectively;
$\mu_1,\mu_2,\mu_3,\mu_4$ are $\mathcal{R}$-linear maps from $A, M,
N, B$ to $B$, respectively.

For any $G\in \mathcal{G}$, we will intensively employ the Jordan
derivation equation
$$
\Theta_{\rm Jord}(G^2)=G\Theta_{\rm Jord}(G)+\Theta_{\rm Jord}(G)G.
\eqno(3.1)
$$
Taking $G=\left[\smallmatrix a & 0\\
0 & 0 \endsmallmatrix \right]$ into $(3.1)$ we have
$$
\begin{aligned}
\Theta_{\rm Jord}(G^2)=& \left[
\begin{array}
[c]{cc}%
\delta_1(a^2) & \tau_1(a^2)\\
\nu_1(a^2) & \mu_1(a^2)\\
\end{array}
\right]
\end{aligned}
\eqno(3.2)
$$
and
$$
\begin{aligned}
&G\Theta_{\rm Jord}(G)+\Theta_{\rm Jord}(G)G= \left[
\begin{array}
[c]{cc}%
a\delta_1(a)+\delta_1(a)a & a\tau_1(a)\\
\nu_1(a)a & 0\\
\end{array}
\right].
\end{aligned}
\eqno(3.3)
$$
By $(3.2)$ and $(3.3)$ we know that $\delta_1$ is a Jordan
derivation of $A$,
$$
\tau_1(a^2)=a\tau_1(a), \quad\quad\nu_1(a^2)=\nu_1(a)a \eqno(3.4)
$$
and
$$
\mu_1(a^2)=0.\eqno(3.5)
$$
for all $a\in A$.
Similarly, putting $G=\left[\smallmatrix 0 & 0\\
0 & b \endsmallmatrix \right]$ in $(3.1)$ gives
$$
\begin{aligned}
\Theta_{\rm Jord}(G^2)=& \left[
\begin{array}
[c]{cc}%
\delta_4(b^2) & \tau_4(b^2)\\
\nu_4(b^2) & \mu_4(b^2)\\
\end{array}
\right]
\end{aligned}
\eqno(3.6)
$$
and
$$
\begin{aligned}
&G\Theta_{\rm Jord}(G)+\Theta_{\rm Jord}(G)G= \left[
\begin{array}
[c]{cc}%
0 & \tau_4(b)b\\
b\nu_4(b) & b\mu_4(b)+\mu_4(b)b\\
\end{array}
\right].
\end{aligned}
\eqno(3.7)
$$
Combining $(3.6)$ with $(3.7)$ yields that $\mu_4$ is a Jordan
derivation of $B$,
$$
\tau_4(b^2)=\tau_4(b)b, \quad\quad\nu_4(b^2)=b\nu_4(b)\eqno(3.8)
$$
and
$$
\delta_4(b^2)=0.\eqno(3.9)
$$
for all $b\in B$.

Let us choose
$G=\left[\smallmatrix 0 & m\\
0 & 0 \endsmallmatrix \right]$ in $(3.1)$. Then
$$
\begin{aligned}
\Theta_{\rm Jord}(G^2)=& \left[
\begin{array}
[c]{cc}%
0 & 0\\
0 & 0\\
\end{array}
\right]
\end{aligned}
\eqno(3.10)
$$
and
$$
\begin{aligned}
&G\Theta_{\rm Jord}(G)+\Theta_{\rm Jord}(G)G= \left[
\begin{array}
[c]{cc}%
m\nu_2(m) & m\mu_2(m)+\delta_2(m)m\\
0 & \nu_2(m)m\\
\end{array}
\right].
\end{aligned}
\eqno(3.11)
$$
The relations $(3.10)$ and $(3.11)$ jointly imply that
$$
m\nu_2(m)=0,
\quad\quad\nu_2(m)m=0 \eqno(3.12)
$$
and
$$
\delta_2(m)m+m\mu_2(m)=0 \eqno(3.13)
$$
for all $m\in M$. Likewise, if we choose
$G=\left[\smallmatrix 0 & 0\\
n & 0 \endsmallmatrix \right]$, then
$$
\begin{aligned}
\Theta_{\rm Jord}(G^2)=& \left[
\begin{array}
[c]{cc}%
0 & 0\\
0 & 0\\
\end{array}
\right]
\end{aligned}
\eqno(3.14)
$$
and
$$
\begin{aligned}
&G\Theta_{\rm Jord}(G)+\Theta_{\rm Jord}(G)G= \left[
\begin{array}
[c]{cc}%
\tau_3(n)n & 0\\
n\delta_3(n)+\mu_3(n)n & n\tau_3(n)\\
\end{array}
\right].
\end{aligned}
\eqno(3.15)
$$
It follows from $(3.14)$ and $(3.15)$ that
$$
n\tau_3(n)=0, \quad\quad \tau_3(n)n=0\eqno(3.16)
$$
and
$$
\mu_3(n)n+n\delta_3(n)=0 \eqno(3.17)
$$
for all $n\in N$. Let us consider
$G=\left[\smallmatrix 1 & m\\
0 & 0 \endsmallmatrix \right]$ in $(3.1)$ and set $\tau_1(1)=m_0$
and $\nu_1(1)=n_0$. Since $\delta_1$ is a Jordan derivation of $A$,
$\delta_1(1)=0$. Moreover, $(3.5)$ implies that $\mu_1(1)=0$.
Therefore
$$
\begin{aligned}
\Theta_{\rm Jord}(G^2)=& \left[
\begin{array}
[c]{cc}%
\delta_2(m) & m_0+\tau_2(m) \\
n_0+\nu_2(m) & \mu_2(m) \\
\end{array}
\right].
\end{aligned}
\eqno(3.18)
$$
On the other hand, from $(3.12)$ and $(3.13)$ we have that
$$
\begin{aligned}
&G\Theta_{\rm Jord}(G)+\Theta_{\rm Jord}(G)G= \left[
\begin{array}
[c]{cc}%
2\delta_2(m)+mn_0 & m_0+\tau_2(m)\\
n_0+\nu_2(m) & n_0m\\
\end{array}
\right].
\end{aligned}
\eqno(3.19)
$$
By $(3.18)$ and $(3.19)$ we arrive at
$$
\delta_2(m)=-mn_0 \quad {\rm and} \quad\mu_2(m)=n_0m \eqno(3.20)
$$
for all $m\in M$. Let us take
$G=\left[\smallmatrix 1 & 0\\
n & 0 \endsmallmatrix \right]$ in $(3.1)$. Applying $(3.16)$ and
$(3.17)$ leads to
$$
\mu_3(n)=nm_0\quad{\rm and}\quad\delta_3(n)=-m_0n \eqno(3.21)
$$
for all $n\in N$. Furthermore, if we choose
$G=\left[\smallmatrix 1 & 0\\
0 & b \endsmallmatrix \right]$ in $(3.1)$, then it follows from
$(3.8)$ and $(3.9)$ that $2\delta_4(b)=0$,
$$
\nu_4(b)=-bn_0\quad {\rm and} \quad\tau_4(b)=-m_0b \eqno(3.22)
$$
for all $b\in B$. Taking
$G=\left[\smallmatrix a & 0\\
0 & 1 \endsmallmatrix \right]$ into $(3.1)$ and using $(3.4)$ and
$(3.5)$ we obtain $2\mu_1(a)=0$,
$$
\tau_1(a)=am_0\quad{\rm and}\quad\nu_1(a)=n_0a \eqno(3.23)
$$
for all $a\in A$. Let us put
$G=\left[\smallmatrix a & m\\
0 & 0 \endsmallmatrix \right]$ in $(3.1)$. Then the relations
$(3.5)$, $(3.19)$ and $(3.23)$ imply that
$$
\begin{aligned}
\Theta_{\rm Jord}(G^2)=& \left[
\begin{array}
[c]{cc}%
\delta_1(a^2)+\delta_2(am) & a^2m_0+\tau_2(am) \\
n_0a^2+\nu_2(am) & n_0am \\
\end{array}
\right].
\end{aligned}
\eqno(3.24)
$$
On the other hand, by the relations $(3.4)$, $(3.12)$, $(3.13)$,
$(3.20)$ and $(3.23)$ we get
$$
\begin{aligned}
& G\Theta_{\rm Jord}(G)+\Theta_{\rm Jord}(G)G\\ =& \left[
\begin{array}
[c]{cc}%
a\delta_1(a)+\delta_1(a)a+amn_0 & a^2m_0+a\tau_2(m)+\delta_1(a)m+m\mu_1(a)\\
n_0a^2+\nu_2(m)a & n_0am\\
\end{array}
\right].
\end{aligned}
\eqno(3.25)
$$
Combining $(3.24)$ with $(3.25)$ yields $\nu_2(am)=\nu_2(m)a$ and
$$
\tau_2(am)=a\tau_2(m)+\delta_1(a)m+m\mu_1(a)
$$
for all $a\in A, m\in M$. Similarly, taking
$G=\left[\smallmatrix a & 0\\
n & 0 \endsmallmatrix \right]$ into $(3.1)$ gives
$\tau_3(na)=a\tau_3(n)$ and
$$
\nu_3(na)=\nu_3(n)a+n\delta_1(a)+\mu_1(a)n
$$
for all $n\in N, a\in A$. Let us choose
$G=\left[\smallmatrix 0 & m\\
0 & b \endsmallmatrix \right]$ in $(3.1)$. We will get
$\nu_2(mb)=b\nu_2(m)$ and
$$
\tau_2(mb)=\tau_2(m)b+m\mu_4(b)+\delta_4(b)m
$$
for all $m\in M, b\in B$. Putting
$G=\left[\smallmatrix 0 & 0\\
n & b \endsmallmatrix \right]$ in $(3.1)$ and employing the same
computational approach we conclude that $\tau_3(bn)=\tau_3(n)b$ and
$\nu_3(bn)=b\nu_3(n)+\mu_4(b)n+n\delta_4(b)$ for all $b\in B, n\in
N$. Finally, let us set
$G=\left[\smallmatrix 0 & m\\
n & 0 \endsmallmatrix \right]$ in $(3.1)$. We have that
$\delta_1(mn)=-\delta_4(nm)+\tau_2(m)n+m\nu_3(n)$ and
$\mu_4(nm)=-\mu_1(mn)+n\tau_2(m)+\nu_3(n)m$ for all $m\in M, n\in
N$.

If $\Theta_{\rm Jord}$ has the form $(\bigstar2)$ and satisfies
conditions $(1)-(7)$, the assertion that $\Theta_{\rm Jord}$ is a
Jordan derivation of $\mathcal{G}$ will follow from direct
computations. We complete the proof of this proposition.
\end{proof}

From now on, we always assume in this section that $M$ is faithful
as a left $A$-module and also as a right $B$-module, but no any
constraint conditions concerning the bimodule $N$. Then we have the
following:

\begin{corollary}\label{xxsec3.3}
Let $\mathcal{G}$ be a $2$-torsion free generalized matrix algebra
over the commutative ring $\mathcal{R}$. An additive map
$\Theta_{\rm Jord}$ form $\mathcal{G}$ into itself is a Jordan
derivation of $\mathcal{G}$ if and only if it has the form
$$
\begin{aligned}
& \Theta_{\rm Jord}\left(\left[
\begin{array}
[c]{cc}%
a & m\\
n & b\\
\end{array}
\right]\right)\\ =& \left[
\begin{array}
[c]{cc}%
\delta_1(a)-mn_0-m_0n & am_0-m_0b+\tau_2(m)+\tau_3(n)\\
n_0a-bn_0+\nu_2(m)+\nu_3(n) & n_0m+nm_0+\mu_4(b)\\
\end{array}
\right] , (\bigstar3)\\
& \forall \left[
\begin{array}
[c]{cc}%
a & m\\
n & b\\
\end{array}
\right]\in \mathcal{G},
\end{aligned}
$$
where $m_0\in M, n_0\in N$ and
$$
\begin{aligned} \delta_1:& A \longrightarrow A,  &  \tau_2:
& M\longrightarrow M, & \tau_3: & N\longrightarrow M,\\
\nu_2: & M\longrightarrow N, & \nu_3: & N\longrightarrow N &  \mu_4:
& B\longrightarrow B
\end{aligned}
$$
are all $\mathcal {R}$-linear maps satisfying conditions
\begin{enumerate}
\item [(1)] $\delta_1$ is a derivation on $A$ and
$\delta_1(mn)=\tau_2(m)n+m\nu_3(n);$

\item[(2)] $\mu_4$ is a derivation on $B$ and
$\mu_4(nm)=n\tau_2(m)+\nu_3(n)m;$

\item[(3)] $\tau_2(am)=a\tau_2(m)+\delta_1(a)m$ and
$\tau_2(mb)=\tau_2(m)b+m\mu_4(b);$

\item[(4)] $\nu_3(na)=\nu_3(n)a+n\delta_1(a)$ and
$\nu_3(bn)=b\nu_3(n)+\mu_4(b)n;$

\item[(5)] $\tau_3(na)=a\tau_3(n)$, $\tau_3(bn)=\tau_3(n)b$,
$n\tau_3(n)=0$, $\tau_3(n)n=0;$

\item[(6)] $\nu_2(am)=\nu_2(m)a$, $\nu_2(mb)=b\nu_2(m)$,
$m\nu_2(m)=0$, $\nu_2(m)m=0.$
\end{enumerate}
\end{corollary}

\begin{proof}
Let $\Theta_{\rm Jord}$ be a Jordan derivation of $\mathcal{G}$.
Then $\Theta_{\rm Jord}$ has the form of $(\bigstar2)$ and satisfies
all additional conditions $(1)-(7)$ of Proposition \ref{xxsec3.2}.
Since $\mathcal{G}$ is a $2$-torsion free generalized matrix
algebra, $\delta_4=0$ and $\mu_1=0$ by condition $(3)$ of
Proposition \ref{xxsec3.2}. Condition $(3)$ of Proposition
\ref{xxsec3.2} vanishes in the present case. Condition $(4)$ of
Proposition \ref{xxsec3.2} correspondingly becomes
$$
\tau_2(am)=a\tau_2(m)+\delta_1(a)m
$$
and
$$
\tau_2(mb)=\tau_2(m)b+m\mu_4(b).
$$
Clearly, we only need to prove that $\delta_1$ is a derivation of
$A$ and that $\mu_4$ is a derivation of $B$. Then for arbitrary
elements $a_1, a_2\in A$, we have
$$
\tau_2(a_1a_2m)=a_1a_2\tau_2(m)+\delta_1(a_1a_2)m \eqno(3.26)
$$
and
$$
\begin{aligned}
\tau_2(a_1a_2m) & =a_1\tau_2(a_2m)+\delta_1(a_1)a_2m\\
& =a_1a_2\tau_2(m)+a_1\delta_1(a_2)m+\delta_1(a_1)a_2m.
\end{aligned}\eqno(3.27)
$$
Combining $(3.26)$ and $(3.27)$ gives
$$
\delta_1(a_1a_2)m=a_1\delta_1(a_2)m+\delta_1(a_1)a_2m.\eqno(3.28)
$$
Note that $M$ is faithful as left $A$-module. Relation $(3.28)$
implies that
$$
\delta_1(a_1a_2)=a_1\delta_1(a_2)+\delta_1(a_1)a_2
$$
for all $a_1, a_2\in A$. So $\delta_1$ is a derivation of $A$.
Similarly, we can show that $\mu_4$ is a derivation of $B$.

Conversely, if an additive map $\Theta_{\rm Jord}$ of $\mathcal{G}$
is of the form $(\bigstar3)$ and satisfies all additional conditions
$(1)-(6)$, then the fact that is a Jordan derivation of
$\mathcal{G}$ will follow from direct computations.
\end{proof}

In view of Herstein's result and recent intensive works
\cite{Bresar1, Bresar2, Bresar4, Johnson, MaJi, Zhang, XiaoWei1,
ZhangYu}, the following question is at hand.

\begin{question}\label{xxsec3.4}
Is each Jordan derivation on a generalized matrix algebra
$\mathcal{G}$ a derivation, or equivalently, do there exist proper
Jordan derivations on generalized matrix algebras?
\end{question}

The following counterexample provides an explicit answer to the
above question. It is shown that Jordan derivations of generalized
matrix algebras need not be derivations. Equivalently, there indeed
exist proper Jordan derivations on certain generalized matrix
algebras.

\begin{example}\label{xxsec3.5}
Let $\mathcal{G}=\left[\smallmatrix A & M\\
N & B \endsmallmatrix \right]$ be a generalized matrix algebra of order $2$
over the commutative ring $\mathcal{R}$. For arbitrary $X=\left[\smallmatrix a_1 & m_1\\
n_1 & b_1 \endsmallmatrix \right]\in \mathcal{G}, Y=\left[\smallmatrix a_2 & m_2\\
n_2 & b_2 \endsmallmatrix \right]\in \mathcal{G}$, we define the sum
$X+Y$ as usual. The multiplication $XY$ is given by the rule
$$
XY=\left[
\begin{array}
[c]{cc}%
a_1a_2 & a_1m_2+m_1b_2\\
n_1a_2+b_1n_2 & b_1b_2\\
\end{array}
\right].\eqno(\spadesuit)
$$
Such kind of generalized matrix algebras are called \textit{trivial
generalized matrix algebras}. That is, the bilinear pairings
$\Phi_{MN}=\Psi_{NM}=0$ are both zero. Let us establish an
$\mathcal{R}$-linear map
$$
\begin{aligned}
\Gamma_{\rm Jord}: \mathcal{G} & \longrightarrow
\mathcal{G}\\
\left[
\begin{array}
[c]{cc}%
a & m\\
n & b\\
\end{array}
\right] & \longrightarrow \left[
\begin{array}
[c]{cc}%
0 & m+n\\
m-n & 0\\
\end{array}
\right], \hspace{3pt} \forall \left[
\begin{array}
[c]{cc}%
a & m\\
n & b\\
\end{array}
\right]\in \mathcal{G}.
\end{aligned}
$$
By straightforward computations, we know that $\Gamma_{\rm Jord}$ is
a Jordan derivation of $\mathcal{G}$, but not a derivation.

On the other hand, we can also define two $\mathcal{R}$-linear maps
$$
\begin{aligned}
\Theta_1: \mathcal{G} & \longrightarrow
\mathcal{G}\\
\left[
\begin{array}
[c]{cc}%
a & m\\
n & b\\
\end{array}
\right] & \longrightarrow \left[
\begin{array}
[c]{cc}%
0 & m\\
-n & 0\\
\end{array}
\right], \hspace{3pt} \forall \left[
\begin{array}
[c]{cc}%
a & m\\
n & b\\
\end{array}
\right]\in \mathcal{G}
\end{aligned}
$$
and
$$
\begin{aligned}
\Theta_2: \mathcal{G} & \longrightarrow
\mathcal{G}\\
\left[
\begin{array}
[c]{cc}%
a & m\\
n & b\\
\end{array}
\right] & \longrightarrow \left[
\begin{array}
[c]{cc}%
0 & n\\
m & 0\\
\end{array}
\right], \hspace{3pt} \forall \left[
\begin{array}
[c]{cc}%
a & m\\
n & b\\
\end{array}
\right]\in \mathcal{G}.
\end{aligned}
$$
It is easy to see that $\Theta_1$ is a derivation of $\mathcal{G}$
and $\Theta_2$ is an anti-derivation of $\mathcal{G}$. Therefore
$\Gamma_{\rm Jord}$ is the sum of the derivation $\Theta_1$ and the
anti-derivation $\Theta_2$.

As a matter of fact, there exist some generalized matrix algebras
whose multiplication satisfies the rule $(\spadesuit)$. Let
$\mathcal{R}^\prime$ be an associative ring with identity and
$\mathcal{Z(R^\prime)}$ be its center. Let us consider the usual
$2\times 2$ matrix ring $
\left[\smallmatrix \mathcal{R}^\prime & \mathcal{R}^\prime\\
\mathcal{R}^\prime & \mathcal{R}^\prime \endsmallmatrix \right]$. It
will become a generalized matrix algebra under the usual addition
and the following multiplication rule
$$
\left[
\begin{array}
[c]{cc}%
a & c\\
d & b\\
\end{array}
\right]\left[
\begin{array}
[c]{cc}%
e & g\\
h & f\\
\end{array}
\right]=\left[
\begin{array}
[c]{cc}%
ae+sch & ag+cf\\
de+bh & sdg+bf\\
\end{array}
\right],
$$
where $s\in \mathcal{Z(R^\prime)}$. A trivial generalized matrix
algebra arises in the case of $s=0$. The usual $2\times 2$ matrix
ring is produced when $s=1$.
\end{example}

In view of Example \ref{xxsec3.5} and our main motivation, we now
begin to describe the forms of anti-derivations on the generalized
matrix algebra $\mathcal{G}$. We will see below, Example
\ref{xxsec3.5} can be lifted and extracted to a more general
conclusion.

\begin{proposition}\label{xxsec3.6}
An additive map $\Theta_{\rm antid}$ from $\mathcal{G}$ into itself
is an antiderivation if and only if it has the form
$$
\begin{aligned}
& \Theta_{\rm antid}\left(\left[
\begin{array}
[c]{cc}%
a & m\\
n & b\\
\end{array}
\right]\right)  \\=& \left[
\begin{array}
[c]{cc}%
0 & am_0-m_0b+\tau_3(n)\\
n_0a-bn_0+\nu_2(m) & 0\\
\end{array}
\right] , (\bigstar4) \\ & \forall \left[
\begin{array}
[c]{cc}%
a & m\\
n & b\\
\end{array}
\right]\in \mathcal{G},
\end{aligned}
$$
where $m_0\in M, n_0\in N$ and
$$
\begin{aligned} \tau_3: & N\longrightarrow M, &
\nu_2: & M\longrightarrow N
\end{aligned}
$$
are $\mathcal{R}$-linear maps satisfying the following conditions:
\begin{enumerate}
\item[{\rm(1)}] $[a, a']m_0=0$, $m_0[b, b']=0$, $n_0[a, a']=0$,  $[b,
b']n_0=0$ for all $a'\in A, b'\in B;$

\item[{\rm(2)}] $m_0n=0$, $nm_0=0$, $mn_0=0$, $n_0m=0;$

\item[{\rm(3)}] $\tau_3(na)=a\tau_3(n)$, $\tau_3(bn)=\tau_3(n)b$,
$n\tau_3(n')=0$, $\tau_3(n)n'=0$ for all $n'\in N;$

\item[{\rm(4)}] $\nu_2(am)=\nu_2(m)a$, $\nu_2(mb)=b\nu_2(m)$,
$m\nu_2(m')=0$, $\nu_2(m)m'=0$ for all $m'\in M$.
\end{enumerate}
\end{proposition}

\begin{proof}
Suppose that the Jordan derivation $\Theta_{\rm antid}$ is of the
form
$$
\begin{aligned}
& \Theta_{\rm antid}\left(\left[
\begin{array}
[c]{cc}%
a & m \\
n & b \\
\end{array}
\right]\right) \\
= &\left[
\begin{array}
[c]{cc}%
\delta_1(a)+\delta_2(m)+\delta_3(n)+\delta_4(b) & \tau_1(a)+\tau_2(m)+\tau_3(n)+\tau_4(b) \\
\nu_1(a)+\nu_2(m)+\nu_3(n)+\nu_4(b) & \mu_1(a)+\mu_2(m)+\mu_3(n)+\mu_4(b) \\
\end{array}
\right] ,
\end{aligned}
$$
for all $\left[\smallmatrix a & m\\
n & b
\endsmallmatrix \right]\in \mathcal{G}$, where $\delta_1,\delta_2,\delta_3,\delta_4$ are
$\mathcal{R}$-linear maps from $A, M, N, B$ to $A$, respectively;
$\tau_1,\tau_2$, $\tau_3,\tau_4$ are $\mathcal{R}$-linear maps from
$A, M, N, B$ to $M$, respectively; $\nu_1,\nu_2,\nu_3,\nu_4$ are
$\mathcal{R}$-linear maps from $A, M, N, B$ to $N$, respectively;
$\mu_1,\mu_2,\mu_3,\mu_4$ are $\mathcal{R}$-linear maps from $A, M,
N, B$ to $B$, respectively.

For any $G_1, G_2\in \mathcal{G}$, we will intensively employ the
antiderivation equation
$$
\Theta_{\rm antid}(G_1G_2)=\Theta_{\rm antid}(G_2)G_1+G_2\Theta_{\rm
antid}(G_1). \eqno(3.29)
$$

Taking $G_1=\left[\smallmatrix a & 0\\
0 & 0 \endsmallmatrix \right]$ and $G_2=\left[\smallmatrix a' & 0\\
0 & 0 \endsmallmatrix \right]$ into $(3.29)$ yields
$$
\begin{aligned}
\Theta_{\rm antid}(G_1G_2)=& \left[
\begin{array}
[c]{cc}%
\delta_1(aa') & \tau_1(aa')\\
\nu_1(aa') & \mu_1(aa')\\
\end{array}
\right]
\end{aligned}
\eqno(3.30)
$$
and
$$
\begin{aligned}
&\Theta_{\rm antid}(G_2)G_1+G_2\Theta_{\rm antid}(G_1)= \left[
\begin{array}
[c]{cc}%
\delta_1(a')a+a'\delta_1(a) & a'\tau_1(a)\\
\nu_1(a')a & 0\\
\end{array}
\right].
\end{aligned}
\eqno(3.31)
$$
It follows from $(3.30)$ with $(3.31)$ that $\delta_1$ is an
antiderivation of $A$, $\mu_1=0$ and
$$
\nu_1(aa')=\nu_1(a')a \eqno(3.32)
$$
for all $a, a'\in A$. Let us set $a'=1$ in $(3.32)$ and denote
$\nu_1(1)$ by $n_0$. Then $\nu_1(a)=n_0a.$ Furthermore, $(3.32)$
implies that $n_0aa'=n_0a'a$ for all $a,a'\in A$, that is, $n_0[a,
a']=0$ for all $a,a'\in A$. If we denote $\tau_1(1)$ by $m_0$, then
we obtain $\tau_1(a)=am_0$ and $[a, a']m_0=0$ for all $a,a'\in A$.

Let us choose
$G_1=\left[\smallmatrix 0 & 0\\
0 & b \endsmallmatrix \right]$ and $G_2=\left[\smallmatrix 0 & 0\\
0 & b' \endsmallmatrix \right]$ in $(3.29)$. By the same
computational approach we conclude that $\mu_4$ is an antiderivation
of $B$, $\delta_4=0$ and
$$
\tau_4(b)=\tau_4(1)b, \quad \nu_4(b)=b\nu_4(1), \quad \tau_4(1)[b,
b']=0, \quad [b, b']\nu_4(1)=0 \eqno(3.33)
$$
for all $b,b'\in B$. We claim that $\tau_4(1)=-m_0$. In fact, this
can be obtained by
taking $G_1=\left[\smallmatrix 0 & 0\\
0 & 1 \endsmallmatrix \right]$ and $G_2=\left[\smallmatrix 1 & 0\\
0 & 0 \endsmallmatrix \right]$ in $(3.29)$. Likewise, we assert that
$\nu_4(1)=-n_0$. Thus the relation $(3.33)$ becomes
$$
\tau_4(b)=-m_0b, \quad \nu_4(b)=-bn_0, \quad m_0[b, b']=0, \quad [b,
b']n_0=0
$$
for all $b,b'\in B$.

Putting $G_1=\left[\smallmatrix 1 & 0\\
0 & 0 \endsmallmatrix \right]$ and $G_2=\left[\smallmatrix 0 & m\\
0 & 0 \endsmallmatrix \right]$ in $(3.29)$ and using the fact
$\mu_1=0$ gives
$$
\begin{aligned}
\Theta_{\rm antid}(G_1G_2)=& \left[
\begin{array}
[c]{cc}%
\delta_2(m) & \tau_2(m)\\
\nu_2(m) & \mu_2(m)\\
\end{array}
\right]
\end{aligned}
\eqno(3.34)
$$
and
$$
\begin{aligned}
&\Theta_{\rm antid}(G_2)G_1+G_2\Theta_{\rm antid}(G_1)= \left[
\begin{array}
[c]{cc}%
\delta_2(m)+mn_0 & 0\\
\nu_2(m) & 0\\
\end{array}
\right].
\end{aligned}
\eqno(3.35)
$$
Combining $(3.34)$ with $(3.35)$ leads to
$$
mn_0=0,\quad \tau_2=0,\quad \mu_2=0
$$
for all $m\in M$. Interchanging $G_1$ and $G_2$ we will get
$$
\delta_2=0,\quad n_0m=0
$$
for all $m\in M$.

If we take
$G_1=\left[\smallmatrix 1 & 0\\
0 & 0 \endsmallmatrix \right]$ and $G_2=\left[\smallmatrix 0 & 0\\
n & 0 \endsmallmatrix \right]$ into $(3.29)$, then
$$
\begin{aligned}
\Theta_{\rm antid}(G_1G_2)=& \left[
\begin{array}
[c]{cc}%
0 & 0\\
0 & 0\\
\end{array}
\right]
\end{aligned}
\eqno(3.36)
$$
and
$$
\begin{aligned}
&\Theta_{\rm antid}(G_2)G_1+G_2\Theta_{\rm antid}(G_1)= \left[
\begin{array}
[c]{cc}%
\delta_3(n) & 0\\
\nu_3(n) & 0\\
\end{array}
\right].
\end{aligned}
\eqno(3.37)
$$
will follow from the fact $\delta_1(1)=0$. By $(3.36)$ and $(3.37)$
we obtain that
$$
\delta_3=0,\quad \nu_3=0.\eqno(3.38)
$$
Interchanging $G_1$ and $G_2$ again yields
$$
\mu_3=0, \quad m_0n=0 \eqno(3.39)
$$
for all $n\in N$. In order to get $nm_0=0$, we only need to put
$G_1=\left[\smallmatrix 0 & 0\\
0 & 1 \endsmallmatrix \right]$ and $G_2=\left[\smallmatrix 0 & 0\\
n & 0 \endsmallmatrix \right]$ in $(3.29)$.

Taking $G_1=\left[\smallmatrix 0 & 0\\
n & 0 \endsmallmatrix \right]$ and $G_2=\left[\smallmatrix a & 0\\
0 & 0 \endsmallmatrix \right]$ into $(3.29)$ and applying $(3.38)$
and $(3.39)$ we arrive at
$$
\begin{aligned}
\Theta_{\rm antid}(G_1G_2)=& \left[
\begin{array}
[c]{cc}%
0 & \tau_3(na)\\
0 & 0\\
\end{array}
\right].
\end{aligned}
\eqno(3.40)
$$
The fact $\mu_1=0$ and $(3.39)$ imply that
$$
\begin{aligned}
&\Theta_{\rm antid}(G_2)G_1+G_2\Theta_{\rm antid}(G_1)= \left[
\begin{array}
[c]{cc}%
0 & a\tau_3(n)\\
0 & 0\\
\end{array}
\right].
\end{aligned}
\eqno(3.41)
$$
The relations $(3.40)$ and $(3.41)$ jointly show that
$\tau_3(na)=a\tau_3(n)$ for all $a\in A, n\in N$.
Likewise, if we choose $G_1=\left[\smallmatrix 0 & 0\\
0 & b \endsmallmatrix \right]$ and $G_2=\left[\smallmatrix 0 & 0\\
n & 0 \endsmallmatrix \right]$ in $(3.29)$, then
$\tau_3(bn)=\tau_3(n)b$ for all $b\in B, n\in N$. The equalities
$\nu_2(am)=\nu_2(m)a$ and $\nu_2(mb)=b\nu_2(m)$ can be obtained by
analogous discussions and the details are omitted here.

Let us consider $G_1=\left[\smallmatrix 0 & 0\\
n & 0 \endsmallmatrix \right]$ and $G_2=\left[\smallmatrix 0 & 0\\
n' & 0 \endsmallmatrix \right]$ in $(3.29)$. Then we get
$n\tau_3(n')=0$ and
$\tau_3(n)n'=0$ for all $n, n'\in N$. Putting $G_1=\left[\smallmatrix 0 & m\\
0 & 0 \endsmallmatrix \right]$ and $G_2=\left[\smallmatrix 0 & m'\\
0 & 0 \endsmallmatrix \right]$ in $(3.29)$ yields $m\nu_2(m')=0$ and
$\nu_2(m)m'=0$ for all $m, m'\in M$.

Taking $G_1=\left[\smallmatrix 0 & m\\
0 & 0 \endsmallmatrix \right]$ and $G_2=\left[\smallmatrix a & m\\
0 & 0 \endsmallmatrix \right]$ into $(3.29)$. Then $\delta_1(a)m=0$
for all $a\in A, m\in
M$. Putting $G_1=\left[\smallmatrix 0 & m\\
0 & b \endsmallmatrix \right]$ and $G_2=\left[\smallmatrix 0 & m\\
0 & 0 \endsmallmatrix \right]$ in $(3.29)$. Then $m\mu_4(b)=0$ for
all $b\in B, m\in M$.  It follows from the faithfulness of $M$ that
$\delta_1=0$ and $\mu_4=0$.

Conversely, suppose that $\Theta_{\rm antid}$ is of the form
$(\bigstar4)$ and satisfies conditions $(1)-(4)$. Then the fact that
$\Theta_{\rm antid}$ is a antiderivation of $\mathcal{G}$ will
follow by direct computations.
\end{proof}

Let us next observe the antiderivations of a class of generalized
matrix algebras.

\begin{definition}\label{xxsec3.7}
Let $\mathcal{G}=
\left[\smallmatrix A & M\\
N & B \endsmallmatrix \right]$ be a generalized matrix algebra
originating from the Morita context $(A, B, _AM_B,$ $ _BN_A,
\Phi_{MN}, \Psi_{NM})$. The bilinear form $\Phi_{MN}: M\underset
{B}{\otimes} N\longrightarrow A$ (resp. $\Psi_{NM}: N\underset
{A}{\otimes} M\longrightarrow B$) is called \textit{nondegenerate}
if for any $0\neq m\in M$ and $0\neq n\in N $, $\Phi_{MN}(m, N)\neq
0$ and $\Phi_{MN}(M, n)\neq 0$ (resp. $\Psi_{NM}(n, M)\neq 0$ and
$\Psi_{NM}(N, m)\neq 0$).
\end{definition}

\begin{example}\label{xxsec3.8}
Let $H$ be a finite dimensional Hopf algebra over filed $\mathbb{K}$
and $A$ be an $H$-module algebra. Let $A^H$ be the invariant
subalgebra of $H$ on $A$, and $A\# H$ be the smash product algebra
of $A$ and $H$. We now consider the generalized matrix algebra
$$
\mathcal{G}_{\rm SPA}=\left[
\begin{array}
[c]{cc}%
A^H & M\\
N & A\#H\\
\end{array}
\right]
$$
defined in Example \ref{xxsec2.1}, where $M=_{A^H}A_{A\#H}$ and
$N=_{A\#H}A_{A^H}$. Suppose that $M$ is a faithful right (or left)
$A\#H$-module. By \cite[Proposition 2.13]{CohenFischmanMontgomery}
we know that the bilinear form $\Phi_{MN}$ will be nondegenerate. In
this case, we easily check that there is indeed no nonzero
antiderivations on $\mathcal{G}_{\rm SPA}$.
\end{example}

\begin{example}\label{xxsec3.9}
Let $\mathbb{K}$ be a field and $A$ be an associative algebra over
$\mathbb{K}$. Let $G$ be a group and $A*G$ be the skew group algebra
over $\mathbb{K}$. Suppose that $A^G$ is the fixed ring of the
action $G$ on $A$. We now revisit the generalized matrix algebra
$$
\mathcal{G}_{\rm GA}=\left[
\begin{array}
[c]{cc}%
A^G & M\\
N & A*G\\
\end{array}
\right]
$$
in Example \ref{xxsec2.2}, where $M=_{A^G}A_{A*G}$ and
$N=_{A*G}A_{A^G}$. For an arbitrary element $n\in N$, we define
$$
n^\bot=\left\{m\in M|\Psi_{NM}(n, m)=0\right\}.
$$
Similarly, for an arbitrary element $m\in M$, we define
$$
m^\bot=\left\{n\in N|\Psi_{NM}(n, m)=0\right\}.
$$
Then $n^\bot$ is a $G$-invariant right ideal of $A$ contained in
$r_A(n)$, where $r_A(n)$ is the right annihilator of $n$ in $A$.
Indeed, let $m\in n^\bot$ and $g\in G$, then $\Psi_{NM}(n,
m^g)=\Psi_{NM}(n, m\cdot g)=\Psi_{NM}(n, m)g=0$. Hence $n^\bot$  is
$G$-invariant, the rest is obvious. Similarly, we can show that
$m^\bot$ is a $G$-invariant left ideal of $A$ contained in $l_A(m)$,
where $l_A(m)$ is the left annihilator of $m$ in $A$.

In particular, if $A$ is a semiprime $\mathbb{K}$-algebra, then
$r_A(n)\neq A$ and $l_A(m)\neq A$. This shows that the bilinear form
$\Psi_{NM}$ is nondegenerate. Furthermore, if we assume that the
module $N$ is faithful as a left $A*G$-module, then the bilinear
form $\Phi_{MN}$ will be also nondegenerate. Indeed, let
$\Phi_{MN}(m, N)=0$ for some $m\in M$. Then, $0=N\cdot \Phi_{MN}(m,
N)=\Psi_{NM}(N, m)\cdot N$. By faithfulness and nondegeneracy of
$\Psi_{NM}$ we deduce that $m=0$. If one of the bilinear pairings
$\Phi_{MN}$ and $\Psi_{NM}$ is nondegenerate, then there is no
nonzero antiderivations on $\mathcal{G}_{\rm GA}$, which is similar
to Example \ref{xxsec3.8}.

In order to ensure the semiprimeness of the $\mathbb{K}$-algebra $A$
and the nondegeneracy of the bilinear forms $\Phi_{MN}$ and
$\Psi_{NM}$, $A$ may be one of the following algebras:
\begin{enumerate}
\item[(1)] the quantized enveloping algebra $U_q(
\mathfrak{sl}_2(\mathbb{K}))$ over the field $\mathbb{K}$,

\item[(2)] the quantum $n\times n$ matrix algebra
$\mathcal{O}_q(M_n(\mathbb{K}))$ over the field $\mathbb{K}$,

\item[(3)] the quantum affine $n$-space $\mathcal{O}_q(\mathbb{K}^n)$ over the field $\mathbb{K}$,

\item[(4)] the double affine Hecke algebra $\widetilde{H}$ over the field $\mathbb{K}$.

\item[(5)] the Iwasawa algebra $\Omega_G$ over the finite field
$\mathbb{F}_p$.
\end{enumerate}
\end{example}

In view of Proposition \ref{xxsec3.6}, Example \ref{xxsec3.8} and
Example \ref{xxsec3.9} we immediately have

\begin{proposition}\label{xxsec3.10}
Let $\mathcal {G}$ be a generalized matrix algebra over the
commutative ring $\mathcal{R}$ and $\Theta_{\rm antid}$ be an
$\mathcal {R}$-linear map from $\mathcal{G}$ into itself. If one of
the bilinear forms $\Phi_{MN}: M\underset {B}{\otimes}
N\longrightarrow A$ and $\Psi_{NM}: N\underset {A}{\otimes}
M\longrightarrow B$ is nondegenerate, then $\Theta_{\rm antid}$ is
an antiderivation of $\mathcal {G}$ if and only if $\Theta_{\rm
antid}=0$.
\end{proposition}

We will end this section by investigating properties of Jordan
derivations of generalized matrix algebras with zero bilinear
pairings. Such kind of generalized matrix algebras draw our
attention, which is due to Haghany's work and Example
\ref{xxsec3.5}. Haghany in \cite{Haghany} studied hopficity and
co-hopficity for generalized matrix algebras with zero bilinear
parings. As you see in Example \ref{xxsec3.5}, those generalized
matrix algebras exactly have zero bilinear pairings.

\begin{theorem}\label{xxsec3.11}
Let $\mathcal{G}$ be a $2$-torsion free generalized matrix algebra
over the commutative ring $\mathcal{R}$. If the bilinear pairings
$\Phi_{MN}$ and $\Psi_{NM}$ are both zero, then every Jordan
derivation of $\mathcal{G}$ can be expressed as the sum of a
derivation and an antiderivation.
\end{theorem}

\begin{proof}
Let $\Theta_{\rm Jord}$ be a Jordan derivation of $\mathcal{G}$. By
Corollary \ref{xxsec3.3} we know that $\Theta_{\rm Jord}$ is of the
form
$$
\begin{aligned}
& \Theta_{\rm Jord}\left(\left[
\begin{array}
[c]{cc}%
a & m\\
n & b\\
\end{array}
\right]\right)\\ =& \left[
\begin{array}
[c]{cc}%
\delta_1(a)-mn_0-m_0n & am_0-m_0b+\tau_2(m)+\tau_3(n)\\
n_0a-bn_0+\nu_2(m)+\nu_3(n) & n_0m+nm_0+\mu_4(b)\\
\end{array}
\right] , (\bigstar3)\\
& \forall \left[
\begin{array}
[c]{cc}%
a & m\\
n & b\\
\end{array}
\right]\in \mathcal{G}.
\end{aligned}
$$
It follows from Proposition \ref{xxsec3.1} and Proposition
\ref{xxsec3.6} that there exist a derivation $\Theta_{\rm d}^\prime$
and an antiderivation $\Theta_{\rm antid}^\prime$ such that
$$
\begin{aligned}
\Theta_{\rm Jord}\left(\left[
\begin{array}
[c]{cc}%
a & m\\
n & b\\
\end{array}
\right]\right) & = \left[
\begin{array}
[c]{cc}%
\delta_1(a)-mn_0-m_0n & am_0-m_0b+\tau_2(m)\\
n_0a-bn_0+\nu_3(n) & n_0m+nm_0+\mu_4(b)\\
\end{array}
\right]\\
& +\left[
\begin{array}
[c]{cc}%
0 & \tau_3(n)\\
\nu_2(m) & 0\\
\end{array}
\right]\\
&=\Theta_{\rm d}^\prime\left( \left[
\begin{array}
[c]{cc}%
a & m\\
n & b\\
\end{array}
\right]\right)+\Theta_{\rm antid}^\prime\left(\left[
\begin{array}
[c]{cc}%
a & m\\
n & b\\
\end{array}
\right]\right)
\end{aligned}
$$
for all $
\left[\smallmatrix a & m\\
n & b \endsmallmatrix \right]\in \mathcal{G}$. This shows that
$\Theta_{\rm Jord}$ can be expressed the sum of a derivation
$\Theta_{\rm d}^\prime$ and an antiderivation $\Theta_{\rm
antid}^\prime$, which is the desired result.
\end{proof}

\begin{example}\label{xxsec3.12}
The Jordan derivation $\Gamma_{\rm Jord}$ constructed in Example
\ref{xxsec3.5} can be expressed as the sum of a derivation
$\Theta_1$ and an antiderivation $\Theta_2$.
\end{example}

\begin{example}\label{xxsec3.13}
Let $\mathcal{R}$ be a $2$-torsion free commutative ring with
identity and $T_n(\mathcal{R})(n\geq 2)$ be the upper (or lower)
triangular matrix algebra over $\mathcal{R}$. Clearly,
$T_n(\mathcal{R})(n\geq 2)$ is a generalized matrix algebra with
zero pairings. In view of Theorem \ref{xxsec3.11}, every Jordan
derivation on $T_n(\mathcal{R})(n\geq 2)$ can be written as the sum
of a derivation and an antiderivation. By \cite[Corollary
1.2]{Benkovic} we assert that the part of antiderivation is zero.
This leads to the fact that every Jordan derivation on
$T_n(\mathcal{R})(n\geq 2)$ is a derivation \cite{ZhangYu}.
\end{example}

\end{document}